\documentclass[a4paper,12pt]{amsart}

\usepackage{amsfonts}
 \usepackage{amssymb}
 \usepackage{amsmath,amsxtra,amsthm} 

\usepackage[usenames]{color}

 \usepackage{enumitem}

\usepackage[mathscr]{eucal}

 \usepackage{dsfont}

\usepackage{wrapfig}

\usepackage[all]{xy}

\usepackage{shuffle}
\usepackage{scalerel}[2016/12/29]

\newtheorem{theorem}{Theorem}
\newtheorem{deff}{Definition}

\newtheorem{proposition}{Proposition}
\newtheorem{example}{Example}

\newtheorem{rem}{Remark}

\newcommand{\mto}{\mapsto}
\newcommand{\bqa}{\begin{eqnarray}}
\newcommand\eqa {\end{eqnarray}}
\newcommand{\beq}{\begin{eqnarray}}
\newcommand{\beqn}{\begin{eqnarray}\nonumber}
\newcommand{\eeq}{\end{eqnarray}}
\newcommand{\be}{\begin{array}}
\newcommand{\ee}{\end{array}}

 \newcommand{\pt}{\partial}

 \newcommand{\od}{{\scriptstyle \odot}}

   \newcommand\vf\varphi

 \newcommand{\Hom}{\mathrm{Hom}}
 \newcommand{\End}{\mathrm{End}}

 \newcommand{\uHom}{\underline{\mathrm{Hom}}}

 \newcommand{\cV}{{\mathcal V}}
 \newcommand{\cH}{{\mathcal H}}

 \newcommand{\cO}{{\mathcal{O}}}
 \newcommand{\cB}{{\mathcal B}}

 \newcommand{\cF}{{\mathcal{F}}}

\newcommand{\cA}{{\mathcal A}}

\newcommand{\cR}{{\mathcal R}}

\newcommand{\oM}{{\overline{M}}}

 \newcommand{\cp}{\mathpzc{p}}

 \newcommand{\R}{{\mathbb R}}
 \newcommand{\Z}{{\mathbb Z}}

 \newcommand{\N}{{\mathbb N}}

\newcommand{\uU}{{\underline{U}}}

\newcommand{\compTens}{\mathbin{\widehat{\otimes}}}
\newcommand{\id}{\mathrm{id}}

  \newcommand{\noi}{{\vskip 2mm\noindent}}

  \def\sll{\mathfrak sl}
  \def\g{{\mathfrak g}}
  \def\gl{{\mathfrak gl}}

  \def\sll{{\mathfrak sl}}
  \def\h{{\mathfrak h}}
  
  \def\m{{\mathfrak m}}
  \def\vu{\vec{u}}
  \def\vv{\vec{v}}

  \def\H{\mathbb H}

   \def\a{\alpha}
   \def\b{\beta}
   
   \def\la{\lambda}
   \def\e{\epsilon}

  \def\sst{\scriptscriptstyle}

 \def\bk{\mathds{k}}

\DeclareFontFamily{OT1}{pzc}{}
\DeclareFontShape{OT1}{pzc}{m}{it}{<-> s * [1.15] pzcmi7t}{}
\DeclareMathAlphabet{\mathpzc}{OT1}{pzc}{m}{it}

\newcommand{\Sym}{\mathrm{Sym}}
\newcommand{\sE}{\mathscr{E}}

\newcommand{\sG}{\mathscr{G}}

\newcommand{\sD}{\mathscr{D}}

\newcommand{\UE}{{\mathscr{U}}}
\newcommand{\sT}{{\mathscr{T}}}

\newcommand{\sJ}{{\mathscr{J}}}

\newcommand{\isomto}{\stackrel{\sim}{\rightarrow}}

\begin{document}

\bibliographystyle{amsplain}


\title[Harish-Chandra pairs]{Various instances of Harish-Chandra pairs}
\author{Alexei Kotov}
\address{Alexei Kotov: Faculty of Science, University of Hradec Králové, Rokitanskeho 62, Hradec Králové
50003, Czech Republic}
\email{oleksii.kotov@uhk.cz}

\author{Vladimir Salnikov}
\address{Vladimir Salnikov: LaSIE  -- CNRS \&  La Rochelle University,
Av. Michel Cr\'epeau, 17042 La Rochelle Cedex 1, France}
\email{vladimir.salnikov@univ-lr.fr}

\subjclass[2010]{
}
\begin{abstract} 
In this paper we address several algebraic constructions in the context of groupoids, algebroids and  $\Z$-graded manifolds. 
We generalize the results of integration of $\N$-graded Lie algebras to the honest $\Z$-graded case and provide some examples of application of the technique based on Harish--Chandra pairs. We extend the construction to the algebroids setting, the main example being the action Lie algebroid.  \\[-2em]

\end{abstract}

\maketitle


\renewcommand{\theequation}{\thesection.\arabic{equation}}

\section*{Introduction}\label{sec:introduction}

Graded / super manifolds, or in some communities colored manifolds, have been extensively studied for several decades. They provide on the one hand a universal description for a lot of classical (differential) geometric structures, and on the other a convenient language 
for applications to gauge theories. A (non-exhaustive) list of related works can be found in references in \cite{DGLG}, where we have first addressed the purely mathematical aspects of graded manifolds in the context of integration of differential graded Lie algebras to differential graded Lie groups. In the categorical setting, i.e. both from objects and morphisms perspective, we have proved the equivalence of categories of differential graded Lie groups and Lie algebras passing by the intermediate step of differential graded Harish--Chandra pairs.

While the result for $\N$-graded Lie groups and algebras was expected and essentially similar to the strategy in the super ($\Z_2$-graded) case, the technicalities were much more subtle. In particular the functional spaces were essentially infinite dimensional, coming not only from smooth functions on the base of a graded manifold but also from formal power series on the graded spaces. 
We have noticed however that most of the construction should be extendable to the honest $\Z$-graded case, i.e. having generators of both positive and negative degrees. A tricky point was to make use of the $\Z$-graded version of the Poincar\'e--Birkhoff--Witt theorem. Back then we did not have the ``tools'' for an elegant description of this situation and decided to address it in a separate paper. 
Those tools were indeed found in \cite{AKVS}, where we have explained how the functional space on the $\Z$-graded manifolds can be constructed by enlarging the space generated by polynomials on non-zero degree generators. And more importantly we introduced a way of intrinsically describing the properties of this functional space using (double) filtrations. 

The approach permitted to prove in that same paper a $\Z$-graded analog of the Batchelor's theorem. It turned out to be also fruitful to study the normal form of differential graded manifolds ($\Z$-graded $Q$-manifolds) in \cite{LKS}.  In this paper we will use similar techniques to explain the properties of group-like objects in the category of $\Z$-graded manifolds. We will in particular address the question of $\Z$-graded Poincar\'e--Birkhoff--Witt theorem for various situations, thus revisiting and generalizing the result of \cite{DGLG}.

The paper is organized as follows:  in section \ref{sec:pbw} we revisit the Poincar\'e--Birkoff--Witt theorem and related results, and in particular point-out the key differences of the $\Z$-graded case. The constructions are afterwards applied for the $\Z$-graded Harish--Chandra pairs (section \ref{sec:HCP}): we consider the even- and the super- subalgebra cases, and end up with the general $\Z$-graded super case. In the short sections \ref{subsec:graLie} and \ref{subsec:grgr} we give a natural definition of semi-formal and global graded Lie groups, and compare them with the one we gave in the categorical setting in \cite{DGLG}. The section \ref{sec:algebroids} is devoted to universal enveloping algebroids; we start with a recollection of facts about Hopf algebroids, mostly known in the literature, and use them for enveloping algebras, the main examples being action groupoids / algebroids and the tangent Lie algebroid. For self-consistency of the paper we have recalled the local structure of $\Z$-graded manifolds, and introduced some notations in the appendix \ref{app:Z-cat}.  We have also recapitulated some of the constructions from \cite{AKVS} there, namely the filtrations (convenient to study functional spaces on graded manifolds) and the $\Z$-graded analog of Batchelor's theorem (for local / global structure of them).


\section{PBW for $\Z$-graded algebras} \label{sec:pbw}

\begin{deff}\label{def:graded_Lie}
A \emph{$\Z-$graded Lie algebra} is a $\Z-$graded vector space over $\bk$ 
\beqn\g =\bigoplus_{i\in\Z}\g_i\eeq together with a super skew-symmetric bilinear operation of degree 0
\beqn
[-,-]\colon \g\otimes \g\to g\, \hspace{3mm}
[x,y]=(-1)^{p(x)p(y)+1}[y,x], \hspace{2mm} x,y\in\g,,
\eeq 
which satisfies the super Jacobi identity
\beqn
\big[x, [y,z]\big]=\big[[x,y],z\big]+(-1)^{p(x)p(y)}
\big[y, [x,z]\big]
\eeq
\end{deff}

\begin{example}\label{ex:graded_Lie}
\mbox{}
\begin{itemize}[leftmargin = 2em]
    \item $\g=\sll_2(\bk)$ with the standard basis $(e,f,h)$, such that
    $[h,e]=e$, $h,f=-f$, $[e,f]=2h$. Define $\g_0=\bk h$, $\g_1=\bk e$, $\g_{-1}=\bk f$. We obtain a $\Z-$graded Lie algebra with all elements of even parity.
    \item Given a Lie super algebra $\g$, define a $\Z-$graded Lie algebra $T[1]\g =\g [1]\oplus \g$, where, as usual, $\g [k]$ is concentrated in degree $-k$, so we have  $\left(T[1]\g\right)_0 =\g$ and
    $\left(T[1]\g\right)_{-1} =\g[1]$. The bracket of elements of $\g$ is being as it was, while $\g$ is acting on $\g[1]$ by the left adjoint representation. The bracket on $\g[1]$ obviously vanish.
    \item For any $\Z-$graded Lie algebra $\g$ of pure even parity define $\Pi T\g =\Pi\g \oplus \g$. First, we canonically extend the grading to the tangent bundle with the reversed parity of fibers; it becomes a $\Z-$graded vector superspace. Second, endow the obtained graded space with the structure of a Lie superalgebra using procedure similar to the one in the previous example. We get a $\Z-$graded Lie superalgebra.  
     \end{itemize}
\end{example}

\begin{rem}
Notice that in all three cases, which we described in Example \ref{ex:graded_Lie}, super parity and $\Z-$grading are compatible in a nontrivial way, i.e. the former is not the reduction modulo 2 of the latter.
\end{rem}

\begin{example}\label{ex:linear_group}
Let $V$ be a $\Z-$graded vector superspace. Consider the Lie superalgebra $\g=\gl (V)$. The grading endomorphism of $V$ provides $\g$ with the structure of a graded Lie algebra. The super parity on $\g$ is the reduction of $\Z-$grading modulo 2 if and only if the same property is holding for $V$.
\end{example}

\begin{deff}\label{def:univ_envelop}
The \emph{universal enveloping algebra} of $\g$ is the unital associative algebra $\UE(\g)$, defined as the quotient algebra $T(\g)/I$, where $T(\g)$ is the tensor algebra of the $\bk-$module $\g$ and $I$ is the two-sided ideal generated by all elements of the form 
$$x\otimes y-(-1)^{p(x)p(y)}y\otimes x-[x,y], \text{ for } x,y\in \g.$$
\end{deff}

\noi There is a canonical map $\iota\colon \g\to \UE(\g)$. Later on we will see that it is an inclusion, which will allow us to identify $\g$ with its image in $\UE(\g)$ under the above map.  

\noi In addition to the (generally non-commutative) multiplication, $\UE(\g)$ is also endowed with: 
\begin{itemize}[leftmargin = 2em]
    \item[--] an even compatible degree $0$ coassociative super cocommutative comultiplication $\Delta\colon \UE(\g)\otimes \UE(\g)\to \UE(\g)$, such that
\beqn
\Delta (ab)=\Delta (a)\Delta (b)
\eeq
for all $a,b\in \UE(\g)$ and a counit $\UE(\g)\to\bk$, which descend from the ones on $T(\g)$. In particular, an element $a$ is \emph{primitive}, i.e. it satisfies
$\Delta (a)=a\otimes 1+ 1\otimes a$ if and only if it belongs to $\iota(\g)$. 
    \item[--] an antipode $S\colon\UE(\g)\to\UE(\g)$, defined such that 
    \beq\label{eq:antipode_UEg}
    S\left(x_1\cdots x_m\right)=(-1)^m x_m\cdots x_1\,,\hspace{2mm} \forall x_1,\ldots,x_m\in\g\,.
   \eeq
\end{itemize}
This data (multiplication, comultiplication, unit, counit, and antipode)
make $\UE(\g)$ into a graded Hopf algebra.

\noi It is worth noticing that, given two linear representations $(R_1, \rho_1)$ and $(R_2, \rho_2)$
of $\UE(\g)$ as associative algebras, the tensor product $R_1\otimes_\bk R_2$ admits a canonical structure of a $\UE(\g)-$module by the formula 
\beqn
\UE(\g)\xrightarrow{\Delta} \UE(\g)\otimes \UE(\g)
\xrightarrow{\rho_1\otimes\rho_2} \End (R_1)\otimes \End(R_2) 
\eeq
The coassociativity of $\Delta$ implies that the $\UE(\g)-$module structure on the triple tensor product does not depend on the order of extension. Note that while the representation of the Lie algebra on the tensor product of two modules is obtained using the Leibniz rule, in the case of a universal enveloping algebra, comultiplication plays the same role. This connection becomes especially noticeable if we take into account the fact that the image of the Lie algebra in its universal enveloping algebra is described by primitive elements for which the identity  $\Delta (a)=a\otimes 1+ 1\otimes a$ holds.

\begin{example}[equivariance of comultiplication]
\label{ex:equivariance_of_comult} Let $\g$ be a graded Lie algebra, consider $\UE(\g)$ as a left $\UE(\g)-$module. Then the (super cocommutative) comultiplication \\
$\Delta\colon \UE(\g)\to \UE(\g)\otimes \UE(\g)$ is a morphism of $\UE(\g)-$modules. It follows from the compatibility between multiplication and comultiplication. Indeed, for any $a,b\in \UE(\g)$, one has \\ $\Delta (ab)=\Delta(a)\Delta (b)$.
\end{example}

\begin{example}[equivariance of multiplication]
\label{ex:equivariance_of_mult}
Let $G$ be a smooth Lie supergroup (i.e. for a regression $\deg(\cdot) \equiv p(\cdot) \in \Z_2$), the Lie superalgebra of which is $\g$. By identifying $\g$ with left-invariant super vector fields on $G$, we endow the algebra of smooth functions\footnote{$\cF(X)$ throughout the paper will denote the algebra of functions on $X$ of appropriate class, depending on the context, for example \emph{smooth} when $X$ is a smooth manifold, or \emph{power series} when $X$ is formal.} $\cF (G)$ in a natural way with the structure of a $\g$-module, and hence also a $\UE(\g)$-module. Then the (supercommutative) multiplication $\cF (G)\otimes \cF (G)\to \cF (G)$ is a $\UE(\g)-$equivariant map. It is easy to see that the last statement is an associative version of the Leibniz rule for differentiation of the product of two functions along left-invariant super vector fields representing $\g$.
\end{example}

\noi Assume that each $\g_k$ admits a well-ordered basis $\{x_{\a_j}\}_{\a_j\in I_k}$, parameterized by a totally ordered set $I_k$, consisting of elements of homogeneous parity and degree simultaneously. Then the whole space $\g$ can be supplied with a well-ordered basis, such that any element of $I_0$ is strictly less than any element of $I_k$ for $k\ne 0$. Define $z_{\a_j}=\iota x_{\a_j}$.

\noi Let $\mathsf{A}$ be either $\emptyset$ or a sequence $\{\a_1, \a_2, \ldots, \a_m\}$, where $\a_j\in I_k$ for some $k$; we say that the length of $\mathsf{A}$ is $0$ or $m>0$, respectively. Let us call $\mathsf{A}$ \emph{admissible} if $m=0$ or $\mathsf{A}$ is increasing, i.e. $a_1\le a_2\le\ldots\le \a_m$, and the multiplicity of each index $\a_j$ is at most one whenever the corresponding $x_{\a_j}$ is odd. Notice that for odd elements on always has $x^2=0$ and then it is sometimes convenient to replace $(\iota x)^2$ by $\frac{1}{2}\iota[x,x]$.

\noi Define an \emph{admissible S-monomial} to be $x_\emptyset =1$ or $x_{\sst\mathsf{A}}=x_{\a_1}   \od x_{\a_2}\od \ldots \od x_{\a_m}\in \Sym (\g)$ for an admissible $\mathsf{A}$ of positive length $m$. Likewise, an \emph{admissible U-monomial} is said to be $z_\emptyset =1$ or $z_{\sst\mathsf{A}}=z_{\a_1}    z_{\a_2} \ldots  z_{\a_m}\in \UE (\g)$ for an admissible $\mathsf{A}$ with $m>0$.

\begin{theorem}[Poincar\'{e}–Birkhoff–Witt]
{\rm (\cite{FHT}, p.286)}\label{thm:PBW}
\mbox{}
\begin{enumerate}[leftmargin = 2em]
    \item The admissible U-monomials are a basis of $\UE(\g)$.
    \item In particular, the linear map $\iota$ is an inclusion and extends to an isomorphism of graded vector spaces $\Sym (\g)\isomto \UE(\g)$.
\end{enumerate}
\end{theorem}

\noindent The proof of the  PBW theorem is rather standard, an un-graded version of it can be found for example in \cite{serre}. The $\N-$ or $\Z-$ graded cases are not conceptually different from each other, and basically mimic the super ($\Z_2-$graded) case (see e.g. a recollection in \cite{FHT}). 
There is however the statement about isomorphism of coalgebras in the context; while it seems to be known to the community, we have not found an accessible proof in the literature. And since precisely in the $\Z$-graded case subtleties may occur, we will give some details of it below, namely deduce it from a more general proposition.  

\begin{proposition}\label{prop:coalgebra_iso}
Let $\g$ be a graded Lie algebra admitting a countable homogeneous basis, $\h$ be a graded Lie  subalgebra of $\g$, and $\mathfrak{m}$ be 
a graded vector space complement to $\h$, i.e. $\g=\h\oplus\mathfrak{m}$ as graded vector spaces. 
Then 
\begin{enumerate}[leftmargin = 2em]
    \item $\UE(\h)\otimes \Sym (\mathfrak{m})$ and $\UE(\g)$ are canonically isomorphic as $\bk-$coalgebras.
    \item This isomorphism respects the natural left $\UE(\h)-$module structures on the corresponding spaces.
\end{enumerate}
\end{proposition}
\begin{proof}
Define 
$\varphi (a\otimes y_1\od \ldots \od y_p)=
a \,\psi (y_1\od \ldots\od y_m)$ for $a\in \UE(\h)$, $y_1, \ldots, y_p\in\m$, where
\beq\label{eq:Weyl_iso}
\psi (y_1\od \ldots\od y_m)=\frac{1}{p!} \sum_{\sigma\in S_n} \e (\sigma)
\iota y_{\sigma(1)}\ldots \iota y_{\sigma (p)}
\eeq
for $y_1, \ldots, y_p\in\g$.
First, we check that $\varphi$ is an isomorphism of vector spaces. Since, by the condition of the proposition, $\g$, as a vector space, admits a countable basis, both subspaces $\h$ and $\m$ have the same property. We choose a basis in each of them and combine them; this will give a basis in the whole space $\g$. Let us order the obtained basis of $\g$ in ascending order so that the vectors from the basis of $\h$ are strictly less than the vectors from the basis of $\m$ (i.e. the whole $\h$ comes before the whole $\m$). Denote the basis constructed by us as $\{x_{\a_j}\}_{\a_j\in I}$. Let $z_{\a_j}=\iota x_{\a_j}$ as above.

\noi
We take into account that for every admissible $\mathsf{A}=\{\a_1, \ldots, \a_m\}$
and $\sigma\in S_m$ one has
\beqn
z_{\a_{\sigma(1)}}  \ldots  z_{\a_{\sigma(m)}}=
\e (\sigma) z_{\mathsf{A}} + \mathrm{linear\; combination\; of\; z_{\mathsf{A'}} \; of\; length\; <m}
\eeq
By Theorem \ref{thm:PBW} the set of admissible monomials $\{z_{\mathsf{A}}\}$ is a basis of $\UE(\g)$, hence so is $\{\psi (x_{\a_1}\od \ldots\od x_{\a_{p}})\}$ 
for all admissible $\{\a_1, \ldots, \a_p\}$, such that the corresponding change of basis is lower unitriangular (with respect to the grading by the monomial length). This immediately proves that $\psi$ is a $\bk-$linear isomorphism.

\noi Now we shall verify that $\psi$ respects the comultiplication. It is easy to see that for every $y_1, \ldots, y_m\in\g$ the following formulas hold true
\beq
\Delta \big( y_1 \od\ldots\od y_m\big) &=&
\sum_{p=0}^m\sum_{\sigma\in Sh(p,m-p)}
\e (\sigma) y_{\sigma (1)}\od \ldots \od y_{\sigma (p)} \otimes y_{\sigma (p+1)}\od \ldots \od y_{\sigma (m)} \\
\Delta \big( \iota y_1 \ldots\iota y_m\big) &=&
\sum_{p=0}^m\sum_{\sigma\in Sh(p,m-p)}
\e (\sigma) \iota y_{\sigma (1)} \ldots  \iota y_{\sigma (p)} \otimes \iota y_{\sigma (p+1)} \ldots  \iota y_{\sigma (m)}\,,
\eeq
where $Sh (p,m-p)=\{\sigma\in S_m \,|\, \sigma (1)<\ldots <\sigma (p)\,,\, \sigma (p+1)<\ldots <\sigma (m) \}$. Therefore
\beqn
\big(\psi\otimes\psi\big)\Delta \big( y_1 \od\ldots\od y_m\big)=
\sum_{p=0}^m
\frac{1}{p!}\frac{1}{(m-p)!}
\sum_{
        \substack{
            \tau\in S_p\times S_{m-p} \\
            \sigma\in Sh(p,m-p)
    }
} 
\e (\sigma)\e (\tau) 
\iota y_{\sigma\tau (1)} \ldots  \iota y_{\sigma\tau (p)} \otimes \\ \nonumber
\otimes \iota y_{\sigma\tau (p+1)} \ldots  \iota y_{\sigma\tau (m)} =
\sum_{p=0}^m
\frac{1}{p!}\frac{1}{(m-p)!} \sum_{\sigma\in S_m}
\e (\sigma) \iota y_{\sigma (1)} \ldots  \iota y_{\sigma (p)} \otimes \iota y_{\sigma (p+1)} \ldots  \iota y_{\sigma (m)}
\eeq
On the other hand
\beqn
\Delta\psi \big( y_1 \od\ldots\od y_m\big)=
\Delta \big(\Sym (y_1, \ldots, y_m)\big)=\\ \nonumber
=
\frac{1}{m!} \sum_{p=0}^m
\sum_{
        \substack{
            \tau\in S_m \\
            \sigma\in Sh(p,m-p)
    }
} 
\e (\sigma)\e(\tau)
\iota y_{\tau\sigma (1)} \ldots  \iota y_{\tau\sigma (p)} \otimes \iota y_{\tau\sigma (p+1)} \ldots  \iota y_{\tau\sigma (m)}
= \\ \nonumber
\frac{1}{m!} \sum_{p=0}^m \big|Sh(p,m-p)\big| 
\sum_{\sigma\in S_m}
\e (\sigma) \iota y_{\sigma (1)} \ldots  \iota y_{\sigma (p)} \otimes \iota y_{\sigma (p+1)} \ldots  \iota y_{\sigma (m)}
\eeq
From $\big|Sh(p,m-p)\big|=\frac{m!}{p!(m-p)!}$ we immediately derive the identity 
\beqn 
\big(\psi\otimes\psi\big)\Delta \big( y_1 \od\ldots\od y_m\big)=
\Delta\psi \big( y_1\od \ldots\od y_m\big)
\eeq
which holds true for all $y_1, \ldots, y_m\in\g$. Thus we conclude that $\big(\psi\otimes\psi\big)\Delta=\Delta\psi $ and, as a corollary, prove the first statement of the proposition. 
 
\noi The second statement of Proposition \ref{prop:coalgebra_iso} follows from the compatibility of the product and coproduct.
$\square$
\end{proof}


\section{Harish--Chandra pairs for $\Z$-graded algebras} \label{sec:HCP}

\noi In this section, we will discuss Harish-Chandra pairs for different situations. Some of the material, including the principle construction of Harish-Chandra pairs, is known in the literature (\cite{Kostant:1975,Vishnyakova:2011}). However, the honest $\Z$-graded case, while being very natural, has not been extensively studied, and we are also not aware of papers where pairs were considered where the corresponding Lie subalgebra is a superalgebra, which is exactly in the spirit of decoupling the degree and the parity we have mentioned before. We will show that one can work with such pairs of this more general type in the same way as with pairs where all elements of the subalgebra have even parity.

\noi Subsection \ref{subsec:formal_group} is a kind of entr\'{e}e to get a taste of the theory. It can be said that we are dealing with Harish-Chandra pairs with zero subalgebra.

\noi The formalism of Harish-Chandra pairs, being applied to Lie supergroups (\cite{Kostant:1975,Vishnyakova:2011}), allows to explicitly construct the corresponding Hopf algebra of functions in terms of the Lie superalgebra $\g=\g_{\underline{0}}\oplus \g_{\underline{1}}$ and the Lie group which integrates the even\footnote{Here we use the underlined subscripts to label the components of the Lie superalgebra, to distinguish it from a $\Z$-graded Lie algebra generated only by elements of degree $0$ and $1$.} part $\g_{\underline{0}}$. In subsection \ref{subsec:even_subalgebra}, we briefly explain this mechanism using a somewhat more general example of a Lie superalgebra, where the corresponding subalgebra is even. Then by a simple remark 
we generalize this approach to the case of Lie super subalgebras and apply it to the case of $\Z-$graded algebras in subsection \ref{subsec:graLie}.


\noi\subsection{Formal groups via Hopf algebras.}\label{subsec:formal_group}
It is known that the universal enveloping algebra of the Lie algebra $\g$, completed by the polynomial filtration, contains a formal group integrating $\g$; in particular, the product of two elements of the form $\exp{x}$ for $x\in \g$ can be calculated using Baker–Campbell–Hausdorff (BCH) formula. To be more detailed, the BCH formula (in the Dynkin's form, \cite{Dynkin:1947}) asserts that
\beq
\exp{\vec{u}}\exp{\vec{v}}=\exp{Z(\vec{u},
\vec{v})}\,
\eeq
where $\vu=\sum\limits_{j\in I} u^{\a_j}x_{\a_j}$, $\vv=\sum\limits_{j\in I} (v)^{\a_j}x_{\a_j}$ and
\beq\label{eq:BCH}
Z(\vu,\vv)=
\sum_{n = 1}^\infty\frac {(-1)^{n-1}}{n}
\sum_{\begin{smallmatrix} r_1 + s_1 > 0 \\ \vdots \\ r_n + s_n > 0 \end{smallmatrix}}
\frac{[ \vec{u}^{\, r_1} \vv^{\,s_1} \vu^{\, r_2} \vv^{\, s_2} \dotsm \vu^{\, r_n} \vv^{\, s_n} ]}{\left(\sum\limits_{j = 1}^{n} (r_j + s_j)\right) \cdot \prod\limits_{i = 1}^{n} r_i! s_i!}\, ,
\eeq
\beqn
[ \vec{u}^{\, r_1} \vv^{\,s_1} \vu^{\, r_2} \vv^{\, s_2} \dotsm \vu^{\, r_n} \vv^{\, s_n}  ] = [ \underbrace{\vu,[\vu,\dotsm[\vu}_{r_1} ,[ \underbrace{\vv,[\vv,\dotsm[\vv}_{s_1} ,\,\dotsm\, [ \underbrace{\vu,[\vu,\dotsm[\vu}_{r_n} ,[ \underbrace{\vv,[\vv,\dotsm \vv}_{s_n} ]]\dotsm]]
\eeq
Decompose vector $Z(\vu,\vv)$ into the basis $\{x_{\a_j}\}_{\a_j\in I}$ of $\g$
\beqn
Z(\vu,\vv)=\sum_{j\in I} Z^{\a_j} (\vu, \vv) x_{\a_j}\,.
\eeq
Considering $(u^{\a_j})$ as coordinates on the formal group and identifying formal power series in variables $(u^{\a_j}, v^{\a_j})$ with the completed tensor product $\bk[[u^{\a_j}]]\hat{\otimes}\bk[[v^{\a_j}]] $, we obtain the comultiplication in the algebra of formal power series in variables $u^{\a_j}$ by the formula
\beq\label{eq:comult_formal}
\Delta (u^{\a_j})=Z^{\a_j} (\vu, \vv) \,.
\eeq
On the other hand, since $\UE(\g)$ is a Hopf algebra, the dual space to it is again a Hopf algebra.\footnote{As above, while defining a bialgebra, we allow completion of the tensor product. In what follows, by $\hat \otimes$ we will denote the completed tensor product, but also sometimes omit the hat sign, when it does not lead to confusion.} Identifying $\UE(\g)$ with $\Sym(\g)$ as graded coalgebras by formula \eqref{eq:Weyl_iso} in Proposition \ref{prop:coalgebra_iso}, we automatically identify the dual space $\UE(\g)^*=\Hom_\bk \big(\UE(\g), \bk\big)$ with graded formal power series of variables $u^{\a_j}$. It is not surprising that the comultiplication on $\UE(\g)^*$ coincides with the one on the formal group, given by formula \eqref{eq:comult_formal}. The last statement becomes almost tautological if we introduce a dual admissible basis in the dual space $\UE(\g)^*$ according to the following rule: we decompose the expression $\exp{\vec{u}}$ into homogeneous components
\beqn
\exp{\vu} =
\sum_{\b_1, \ldots, \b_m}
\frac{1}{m!}\, \psi \left( x_{\b_1}\od\cdots\od x_{\b_m}
\right) u^{\b_m}\cdots u^{\b_1}
=\\ \nonumber
=\sum_{\substack{
            \mathsf{B}=\{\b_1, \ldots, \b_m\}- \\
            \mathrm{admissible}
           }           }
\frac{1}{\scriptstyle|\mathsf{B}|!}
\psi \left( x_{\b_1}\od\cdots\od x_{\b_m}
\right) u^{\b_m}\cdots u^{\b_1}\,,
\eeq
where $|\mathsf{B}|! = d_1!\cdots d_r!$ for
$\mathsf{B} = \{\underbrace{\a_1, \ldots, \a_1}_{d_1},
\ldots, \underbrace{\a_r, \ldots, \a_r}_{d_r}
\}$,
and call the monomials $\frac{1}{|\mathsf{B}|!}u^{\b_m}\cdots u^{\b_1}$ admissible whenever so is the sequence $\mathsf{B}$. Then it is obvious that any element from the $\UE(\g)^*$ expands into a formal infinite linear combination of those monomials, and the resulting correspondence between admissible bases in $\UE(\g)$ and $\UE(\g)^*$ is canonical.

\begin{rem}\label{rem:formal_neighborhood}
Assume we are able to integrate a Lie (super)algebra $\g$ into a smooth Lie (super)group $G$. Then the formal group, constructed out of $\g$ in subsection \ref{subsec:formal_group} in terms of the corresponding Hopf algebra of formal power series, can be viewed as the formal neighborhood of the identity in $G$.
\end{rem}


\noindent\subsection{Harish-Chandra pairs 
as a  ``semi-formal'' integration.}\label{subsec:even_subalgebra}
Let us consider a more general situation. Assume there is a pair $(\g,\h)$ of a Lie algebra $\g$ and a Lie subalgebra $\h\subset \g$. Let us also assume that we have succeeded in integrating a subalgebra $\h$ into a Lie group $H$ simultaneously with the adjoint action of $\h$ on $\g$; we call $(H, \g)$ a \emph{Harish-Chandra pair}. Given such a pair, one constructs a  
Hopf algebra $\cA$, 
which represents a ``semi-formal'' (in the sense of \cite{AKVS}) group integrating $\g$ (see also Proposition \ref{prop:formal_neighborhood_subgroup}). By definition,
\beq\label{eq:Harish-Chandra_Hopf}
\cA = \Hom_{\UE(\h)}\left(
\UE(\g), \cF(H)\right)\,,
\eeq
where the action of $\UE(\h)$ on functions on $H$ is generated by left-invariant super vector fields on $H$. The multiplication of two $\bk-$linear morphisms $\Phi_1$ and $\Phi_2$ in $\Hom_{\bk}\left(
\UE(\g), \cF(H)\right)$ is determined as the following composition of maps
\beq
\UE(\g)\xrightarrow{\Delta}\UE(\g)\otimes \UE(\g)\xrightarrow{\Phi_1\otimes \Phi_2} \cF (H)\otimes \cF (H)\xrightarrow{\mu} \cF (H)\,,
\eeq
where $\mu$ is the multiplication of smooth functions on $H$. Since both maps $\Delta$ and $\mu$ are $\UE(\g)-$invariant (see Examples \ref{ex:equivariance_of_comult} and \ref{ex:equivariance_of_mult}), the resulting map descends to the space of $\UE(\h)-$invariants. 

\noi  
The comultiplication in $\cA$, being regarded as a $\UE(\h)\otimes \UE(\h)-$equivariant map from $\cA$ to bi-linear functionals on $\UE (\g)$ with values in $\cF(H\times H)$, is given for any $\phi\in\cA$ by the following formula:
\beq
\Delta \left(\phi\right)(h_1, a_1, h_2, a_2)=
\phi (h_1 h_2, Ad_{h_2}^{-1}(a_1)a_2)\,,
\eeq
where $h_1, h_2\in H$, $a_1, a_2\in \UE(\g)$. The equivariance property \eqref{eq:Harish-Chandra_Hopf} for $\phi$ can be reformulated as
\beqn
\phi (h, za)=\left(\frac{\pt}{\pt \lambda}\right)_{\la=0}
\phi (h\exp{\la z},a)
\eeq
for any $h\in H$, $z\in \h$, $a\in \UE(\g)$. We have to check that $\Delta (\phi)$ obeys the same property with respect to both pairs of arguments $(h_1, a_1)$ and $(h_2, a_2)$. Indeed, for any $z\in \h$ one has
\beqn
\Delta \left(\Phi\right)(h_1, za_1, h_2, a_2)=
\phi (h_1 h_2, Ad_{h_2}^{-1}(za_1)a_2)=
\phi (h_1 h_2, Ad_{h_2}^{-1}(z)Ad_{h_2}^{-1}(a_1)a_2)=
\\ \nonumber
\left(\frac{\pt}{\pt \lambda}\right)_{\la=0}
\phi (h_1 h_2exp{Ad_{h_2}^{-1}(\la z)}, Ad_{h_2}^{-1}(a_1)a_2)=
\left(\frac{\pt}{\pt \lambda}\right)_{\la=0}
\phi (h_1 (exp{\la z})h_2, Ad_{h_2}^{-1}(a_1)a_2)=
\\ \nonumber
\left(\frac{\pt}{\pt \lambda}\right)_{\la=0}
\Delta \left(\phi\right)(h_1 exp{\la z}, a_1, h_2, a_2)
\hspace{10em}
\eeq
Similarly, 
\beqn
\Delta \left(\phi\right)(h_1, a_1, h_2, z a_2)=
\phi (h_1 h_2, Ad_{h_2}^{-1}(a_1)za_2)=
\phi (h_1 h_2, z Ad_{h_2}^{-1}(a_1)a_2)+
\\ \nonumber
\phi (h_1 h_2, [Ad_{h_2}^{-1},z](a_1)a_2)=
\left(\frac{\pt}{\pt \lambda}\right)_{\la=0}
\left(
\phi (h_1 h_2, Ad_{h_2\exp{\la z}}^{-1}(a_1)a_2)
\right)+ \\ \nonumber
\left(
\phi (h_1 h_2 \exp{\la z},Ad_{h_2}^{-1}(a_1)a_2 )
\right)=\left(\frac{\pt}{\pt \lambda}\right)_{\la=0}
\phi (h_1 h_2 \exp{\la z},Ad_{h_2\exp{\la z}}^{-1}(a_1)a_2 )= \\ \nonumber
\left(\frac{\pt}{\pt \lambda}\right)_{\la=0}
\Delta \left(\phi\right)(h_1, a_1, h_2\exp{\la z},  a_2)
\hspace{7em}
\eeq
An antipode $S$ acts on $\phi\in\cA$ as follows:
\beq\label{eq:antipode_HC}
S(\phi)(h,a)=\phi\left(h^{-1}, Ad_h (a)\right)\,, \hspace{2mm} \forall h\in H, a\in\UE(\g)\,.
\eeq
In the same way as it was shown above that comultiplication preserves the $\UE(\h)$-equivariance property, one proves that $S$ obeys the same rule, i.e. it preserves $\cA$.

\begin{rem}\label{rem:HC_with_complement}
Let $\m$ be a vector space complement to $\h$ in $\g$.
By Proposition \ref{prop:coalgebra_iso}, there is a canonical isomorphism of coalgebras over $H(\h)$
\beq\label{eq:iso_coalgebras_over_UEh} 
\Hom_{\bk}\left(
\UE(\g), \cF(H)\right)\simeq \Hom_{\bk}\left(
\UE(\h)\otimes \Sym (\m), \cF(H)\right)
\,,
\eeq
therefore $\cA$ is canonically isomorphic to $\Hom_{\bk}\left(\Sym (\m), \cF (H)\right)$. This will immediately show that the obtained Hopf algebra is isomorphic as an algebra to functions on the ``semi-formal'' manifold $H\times \m$, which are smooth on $H$ and formal with respect to the linear coordinates on $\m$.
\end{rem}

\begin{proposition} \label{prop:formal_neighborhood_subgroup}
Let $\g$ be a finite-dimensional even Lie algebra. Let us integrate $(\g, \h)$ into a pair of a Lie group $G$ and its Lie subgroup $H$. Then the ``semi-formal'' manifold mentioned in Remark \ref{rem:HC_with_complement}, determined by the Harish-Chandra pair $(H,\g)$ is isomorphic to the formal neighborhood of $H\subset G$.
\end{proposition}
\begin{proof}
To prove the assertion of the Proposition, we present an explicit epimorphism of Hopf algebras $\cF (G)\to \cA$, where $\cA$ is as in equation \eqref{eq:Harish-Chandra_Hopf} for the Harish-Chandra pair $(H, \g)$: given any $\tilde\phi\in\cF(G)$, define $\phi\in\cA$, such that
\beqn
\phi (h,x_1\cdots x_m)=\left(\frac{\pt}{\pt\la_1}\right)_{\la_1=0}\cdots \left(\frac{\pt}{\pt\la_m}\right)_{\la_m=0}
\tilde\phi \left(h, he^{\la_1 x_1}\cdots e^{\la_m x_m}
\right)
\eeq
for all $h\in\H$, $x_1,\ldots, x_m\in\g$. The surjectivity of the map $\tilde\phi\mto\phi$ is beyond doubt. The verification that we have obtained a morphism of Hopf algebras follows from the definition of the algebra $\cA$, which, frankly speaking, is arranged in such a way that the map defined above ``respects'' the structure of Hopf algebras. $\square$
\end{proof}

\begin{example}[Integration of a Lie superalgebra by use of Harish-Chandra pairs]\label{ex:Harish-Chandra_for_Lie_super}
Let $\g=\g_{\underline{0}}\oplus\g_{\underline{1}}$ be a finite-dimensional Lie superalgebra. Regard the decomposition of $\g$ into even and odd parts as a reductive decomposition, i.e. let $\h=\g_{\underline{0}}$,
$\m=\g_{\underline{1}}$. Then one can make $(\h,\m)$ into a Harish-Chandra pair by integrating the Lie algebra $\g_{\underline{0}}$ together with its linear action on $\g_{\underline{1}}$. The obtained semi-formal manifold is isomorphic to a Lie supergroup integrating $\g$; given that linear coordinates on the vector space complement $\m=\g_{\underline{1}}$ to $\g_{\underline{0}}$ are odd and thus nilpotent, the corresponding formal power series contain only a finite number of non-zero terms, and are the same as smooth functions of $\m$. It was proved \cite{Kostant:1975,Vishnyakova:2011} that the category of Harish-Chandra pairs of the type described here is equivalent to the category of Lie supergroups.
\end{example}

\noindent The Hopf algebra, constructed out of a Harish-Chandra pair, is a left-
 (right-) $\g-$module, where the corresponding structure 
 for any $\Phi\in \cA =\Hom_{\UE(\h)}\left(
\UE(\g), \cF(H)\right)$, $z\in \g$, $a\in \UE(\g)$, and $h\in H$
 is given by the formulas
 \beq
 \left(x^l\Phi\right)(h,a) &=& \Phi (h, ax) \\
 \left(x^r\Phi\right)(h,a) &=& \Phi (h, Ad^{-1}_h(x)a)
 \eeq


 \begin{rem}[The super subalgebra case]
\label{rem:super_subalgebra} Let $(\g, \h)$ be a pair consisting of a Lie superalgebra and its subalgebra which is not necessarily even. Then it turns out that it is sufficient to apply the Harish-Chandra technique to the even part of $\h$ thanks to the canonical isomorphism
\beq
\Hom_{\UE(\h)}
\left(
\UE(\g), 
\Hom_{\UE(\h_{\underline{0}})}
\left(
\UE(\h), \cF (H_{\underline{0}})
\right)
\right)
\simeq
\Hom_{\UE(\h_{\underline{0}})}
\left(
\UE(\g), \cF (H_{\underline{0}})
\right)
\eeq
    
 \end{rem}



\subsection{Integration of graded Lie algebras} \label{subsec:graLie} A graded Lie algebra (Definition \ref{def:graded_Lie}) is an excellent example of a reductive superalgebra, where the set of elements of degree zero is chosen as a subalgebra $\h$, and a direct sum of subspaces consisting of homogeneous elements of nonzero degree as a complement $\m$:
\beqn
\h=\g_0\,, \hspace{2mm} \m=\bigoplus_{i\ne 0} \g_i
\eeq
We now use the isomorphism \eqref{eq:iso_coalgebras_over_UEh} that respects the comultiplication. 
Following the approach to $\Z$-graded manifolds from \cite{AKVS} (see also the Appendix \ref{app:Z-cat}), we endow $\cF (H)\otimes \Sym (\m)$, viewed as a $\cF(H)-$coalgebra, with a canonical filtration by coideals as in \eqref{eq:cofiltration}. Then, on the dual space $\Hom_{\cF(H)}\left(\cF (H)\otimes \Sym (\m), \cF (H)\right)$, which is isomorphic as an algebra to $\cA$ (defined in \eqref{eq:Harish-Chandra_Hopf}), a canonical filtration by graded ideals arises. This defines the structure of a semi-formal graded manifold on the ``spectrum'' of $\cA$, i.e. on $H\times\m$.



\subsection{Graded groups -- global description} \label{subsec:grgr}

In \cite{DGLG} we have defined \emph{graded Lie groups} as monoid objects (in the category of graded manifolds) which are groups.
The definition works verbatim for any grading: $\Z_2, \N$ or $\Z$, but now knowing the structure of $\Z$-graded manifolds (cf. Appendix \ref{app:Z-cat}) we can be more explicit about the constructions and also explain the subtleties of the integration procedure.

Let $\g$ be a finite dimensional graded super Lie algebra; as before with the degree $deg(\cdot)$ compatible with the parity $p(\cdot)$, but $p(\cdot)$ \emph{is not necessarily} $(deg(\cdot)\!\! \mod 2)$. Decompose it to parity even and odd parts:
$$\g=\g_{\underline{0}}\oplus \g_{\underline{1}}.$$
Let us forget for a second about the grading and integrate $\g$ to a supergroup $\underline G$, using the approach from \cite{Kostant:1975, Vishnyakova:2011} or any other appropriate technique. The result can be viewed using the super (i.e. $\Z_2$-graded) Harish--Chandra pair:
$\underline G \simeq (G_0, \g)$, in a straightforward way for the \cite{Kostant:1975, Vishnyakova:2011}  method, and after some work for others.

Now consider the gradings on $\g$, the Lie algebra part of this pair. It can be viewed as a differentiation -- the linear (Euler) vector field $\epsilon_\g$: as usual this differentiation has integer eigenvalues and homogeneous functions lie in corresponding eigenspaces. And this is a general fact: gradings are in one-to-one correspondence with such differentiations (see e.g. \cite{grab-hom} and references therein).

\noi Let us ask ourselves the question, what object will be induced on the Lie supergroup by means of the derivation of its Lie algebra, given by the linear Euler vector field?

\noi The answer to this question is the Van Est correspondence between cocycles on a Lie group and on its Lie algebra (\cite{vanest}). The Van Est correspondence always works in one direction, namely, each group $p$-cocycle corresponds to a $p$-cocycle on the Lie algebra (for this, the differential of the group cocycle at the identity is used). Conversely, for each Lie algebraic cocycle to correspond to a group cocycle, the group must be simply connected. Given that a derivation of a Lie algebra can be viewed as a 1-cocycle on this algebra with values in the adjoint representation, the Van Est correspondence produces a 1-cocycle on the corresponding simply connected Lie group also with values in the adjoint representation of the Lie group on its Lie algebra; such a 1-cocycle is transformed by left translations into a multiplicative vector field on the group (cf.~ \cite{DGLGa} which is the extended arxiv version of \cite{DGLG}).

\noi Recall that a \emph{multiplicative vector field} $X$ on a graded group $G$ is a vector field compatible with multiplication $m \colon G  \times G \to G$, in a sense that $(X, X)$ is $m$-related to $X$.
In contrast to what is described in \cite{DGLG} we are not integrating a homological odd vector field (nilpotent differentiation), but an even one $\epsilon_\g$. Thus, when the even part of $\underline G$ is simply connected, this gives a unique multiplicative vector field $X$ on $\underline G$. 
\begin{rem}
 We have denoted the graded group by $\underline G$ to stress the fact that the result of integration came initially from the parity and consideration of associated superalgebras, but we will drop the underlining from now on.     
\end{rem}

It is then natural to give the following:
\begin{deff}[$\Z$-graded Lie supergroup] \label{def:ZgraLie}
We call a \emph{$\Z$-graded Lie group} a Lie supergroup $G$ equipped with a multiplicative even vector field $X$, s.t. the corresponding differentiation $\epsilon_{\g}$ defines a $\Z$-grading (in the above sense) on the corresponding Lie algebra $\g$. 
\end{deff}

\begin{rem}
   By a \emph{$\Z$-graded Lie supergroup} we mean \emph{global $\Z$-graded Lie supergroup} in contrast to
   a \emph{local} one described above in \ref{subsec:graLie} or in \cite{DGLG} -- the difference will be clear below. 
\end{rem}

\begin{proposition}\label{prop:gr_Lie_group_locus}
The structure of this multiplicative vector field $X$ on $G$ around its zero locus is similar to the structure of the linear vector field $\epsilon_\g$ on the Lie algebra $\g$, i.e. $X$ is an Euler vector field with the same eigenvalues. 
\end{proposition}
\begin{proof}
For simplicity, let us prove the assertion of the proposition for the case of an ``even'' Lie group. The proof for the ``super'' case will not be fundamentally different, except that one will need to use the notion of a superpoint (or work with the Hopf algebra of functions on a supergroup in the spirit of \cite{DGLG}).

\noi First, note that local integration of a multiplicative vector field $X$ on a Lie group $G$ gives us a local automorphism of this group in a neighborhood of the identity, the infinitesimal counterpart of which is the corresponding derivation $\e_\g$ in the Lie algebra $\g$ (see the Van Est correspondence). This means that the exponential map, which diffeomorphically identifies some neighborhood $U_0$ of zero in the Lie algebra with a neighborhood $U_e$ of identity in the group, $\exp\colon \left(U_0,\e_\g\right)\isomto \left(U_e, X\right)$, takes the linear Euler field $\e_\g$ on the Lie algebra to the above mentioned multiplicative vector field $X$ on the group. This argument shows that there is a neighborhood of the identity in the group in which $X$ has the form of an Euler vector field; moreover, this neighborhood admits homogeneous (for the multiplicative field $X$) coordinates obtained as a pullback of (some) linear homogeneous coordinates on the Lie algebra, with the same weights as those linear coordinates.

\noi Recall (cf.~\cite{DGLGa}) that the multiplicativity of $X$ is equivalent to $X_{gh}=g^l (X_{h})+h^r (X_{g})$ for all $g,h\in G$, where $g^l$ and $h^r$ are the left- and right- translations by $g$ and $h$, respectively. Assume now that $g$ belongs to the zero locus of $X$, i.e. $X_g=0$. Then the multiplicativity turns into the left translation property $X_{gh}=g^l (X_{h})$ for any $h\in G$. The latter implies that $g^l$ diffeomorphically maps the neighborhood $U_e$ of the identity, constructed above, to the neighborhood $U_g=g^l\left(U_e\right)$ of $g$ and, thanks to the left translation property for $X$, it ``commutes'' with $X$: $g^l\colon \left(U_e, X\right)\isomto \left(U_g, X\right)$. This proves that the restriction of $X$ onto $U_g$ will again admit homogeneous coordinates (with the same integer weights as the linear coordinates on $\g$), so $X_{|U_g}$ will also have the form of an Euler vector field.
$\square$
\end{proof}

\begin{rem}
It would be interesting to compare the definition \ref{def:ZgraLie} with the $\Z$-graded analog of homogeneity structures from \cite{grab-hom, gra-gra-ra}.
\end{rem}

It is also interesting to compare this definition with the one above in the categorical language (\cite{DGLG} and section \ref{sec:HCP}). 
In fact the following proposition holds true: 

\begin{proposition} Let $G$ be a graded Lie group in the sense of Definition \ref{def:ZgraLie}. Then the following statements are true:
\begin{itemize}[leftmargin = 2em]
    \item The zero locus $G_0$ of the multiplicative grading vector field $X$ is a Lie subgroup of $G$;
    \item The formal neighborhood of $G_0$ in $G$ is isomorphic to the semi-formal group determined by the Harish-Chandra pair $(\g,G_0)$, i.e. to the result of the integration described in Subsection \ref{subsec:graLie}. This isomorphism is homogeneous of degree zero, i.e. equivariant w.r.t. the Euler vector fields.  
\end{itemize}
\end{proposition}
\begin{proof} As before, for simplicity, we assume that $G$ is an ``even'' Lie group.
According to Proposition \ref{prop:gr_Lie_group_locus}, for each point of the zero locus of the vector field $X$, there is a neighborhood of this point, such that in this neighborhood the field $X$ has the form of an Euler field. From this it is easy to conclude that the condition on the zero locus $G_0$ is described as a joint set of zeros of some subset of local homogeneous coordinates, which means that the zero locus of $X$ is a smooth submanifold of $G$. It follows from the multiplicativity condition for $X$ that $G_0$ is closed under group operations, and thus is a Lie subgroup of $G$,  which proves the first assertion of the proposition. 

\noi The second statement is a consequence of the Proposition \ref{prop:formal_neighborhood_subgroup} for the Harish-Chandra pair $(G_0, \g)$ and Proposition \ref{prop:gr_Lie_group_locus}.
 $\square$
\end{proof}

\section{Universal enveloping algebroids 
}
\label{sec:algebroids}
The idea of applying the similar integration technique to algebroids sounds very natural and may seem straightforward. And as we will see below, some constructions indeed ``work out of the box'', there are however important subtleties to take into account which we will point out.

\noi\subsection{Commutative Hopf algebroids.}\label{subsec:Hopf} In this subsection, we give or recall the necessary information about commutative (but in general non-cocommutative) Hopf algebroids, which reproduce in an algebraic context the properties of an algebra of functions on a groupoid and are related to commutative Hopf algebras in the same way that Lie groupoids are related to Lie groups. There is also a theory of non-commutative Hopf algebroids (\cite{Lu:1996, ping}), which generalize in the same direction the properties of (quantum) universal enveloping algebras. While for a Hopf algebra over a field, the dual vector space also has the structure of a Hopf algebra, up to some modifications in infinite dimension (for the correct definition of a product on the dual space in the topological case, one should consider a topologically completed tensor product\footnote{Already having in mind smooth functions, one realizes that the completion is necessary. And for notations again, in the diagrams the hat in $\compTens$ will sometimes be omitted.}, in the algebraic case, the so called ``finite dual space''), for Hopf algebroids the duality is less trivial and the definitions of dual algebraic structures are different. Below we will see this using the example of a universal enveloping algebroid.

\begin{deff}[Hopf algebroid]\label{def:Hopf_algebroid}
A topological \emph{graded Hopf algebroid} over a commutative ring\footnote{In this article, a commutative ring is always a field of characteristic 0. } $\bk$
is a pair of (Fr\'echet graded) commutative associative 
unital $\bk-$algebras $(\cR,\cH)$ with the structure maps:
\begin{enumerate}[leftmargin = 2em]
\itemsep0em
\item
a left unit/source map $\eta_L \colon \cR \to \cH$;
\item
a right unit/target map $\eta_R \colon \cR \to \cH$;
\item
a comultiplication/composition map $\Delta \colon \cH \to \cH \compTens_\cR \cH$, where the (completed) tensor product is taken with respect to the left- and right- $\cR-$module structures, determined by $\eta_L$ and $\eta_R$, respectively;
\item
a counit/identity map $\varepsilon \colon \cH \to \cR$;
\item
an antipode/conjugation/inverse map $S \colon \cH \to \cH$,
\end{enumerate}
satisfying the following axioms:
\begin{enumerate}[leftmargin = 2em]
\itemsep0em

\item
counit laws
$(\id_\cH \otimes \varepsilon) \circ \Delta = (\varepsilon \otimes \id_\cH) \circ \Delta = \id_\cH$, $\varepsilon\eta_L=\varepsilon\eta_R=\id_\cR$;
\item source/target laws $\Delta\circ \eta_L = \eta_L\otimes 1$, $\Delta\circ \eta_R = 1\otimes \eta_R$;
\item
coassociativity
$(\id_\cH \otimes \Delta) \circ \Delta = (\Delta \otimes \id_\cH) \circ \Delta$;
\item the comultiplication is a morphism of the unital algebras;
\item
the antipode identities
$\mu \circ (\id_\cH \otimes S) \circ \Delta =  \eta_L \circ \epsilon$,
$\mu \circ (S \otimes\id_\cH) \circ \Delta = \eta_R \circ \epsilon$, where $\mu$ is the multiplication in $\cH$.
\end{enumerate}

\noi The last property (of the antipode) can be expressed by the following commutative diagram

\beq\label{eq:diagram_Hopf}
     \xymatrix{
      \cH & \ar[l]_{\id\cdot S} \cH\otimes \cH \ar[d]\ar[r]^{S\cdot\id} & \cH \\
      & \cH\otimes_\cR \cH\ar@{-->}[ur]\ar@{-->}[ul] & \\
      \cR\ar[uu]^{\eta_L} & \cH\ar[u]_\Delta
      \ar[l]_\e \ar[r]^\e & \cR\ar[uu]_{\eta_R}
             }
 \eeq

\noindent where $\id\cdot S = \mu \circ (\id_\cH \otimes S)$, $S\cdot\id = \mu \circ (S \otimes\id_\cH)$.

\noi By omitting the antipode structure together with the antipode identity we obtain the notion of a topological \emph{unital and counital graded bialgebroid}.
\noindent By dropping off the unit and the counit structures and the related identities, we obtain the general notion of a topological \emph{graded bialgebroid}.
\end{deff}

\begin{rem}
One notices easily that the above definition is a natural generalization (``oidification'') of the definition of graded Hopf algebra recalled in \cite{DGLG}, for the (graded) commutative associative case.
\end{rem}

\noi The definition of a Hopf algebroid dualizes the notion of a groupoid in the following sense: for any associative commutative ring $\cB$ over $\bk$, $\Hom_\bk (\cH,\cB)$ and $\Hom_\bk (\cR,\cB)$ gives us the set of morphisms and objects of a groupoid, respectively; we will say that this groupoid structure is parameterized by $\cB$. The counit laws imply that the dual counit map $\varepsilon^*$ acts as a two-sided identity for the groupoid parameterized by a commutative $\bk-$ring. Coassociativity corresponds to the associativity of the composition of morphisms. 


\begin{example}[Action groupoid]\label{ex:action_grpd}
Let $G$ be a Lie group acting on a manifold $M$. Denote the action map by $a$, i.e. $a\colon \cF (M)\to \cF(M\times G)$, such that $a(z,g)=g^{-1}z$ for any $g\in G$, $z\in M$. Introduce the following notations: 
$\cR=\cF (M)$ and $\cH=\cF(M\times G)=\cF(M) \compTens_\bk\cF (G)$. The structure maps of the corresponding Hopf algebroid $\cH$ are:
\begin{itemize}[leftmargin = 2em]
    \item[-] left unit $\eta_L\colon \cR\to\cH$, $\eta_L (f)(z,g)=f(z)$, i.e. $\eta_L=\id\otimes_\bk 1$;
    \smallskip
    \item[-]  right unit $\eta_L\colon \cR\to\cH$,
    $\eta_R (f)(z,g)=f(g^{-1}z)$, i.e. $\eta_R=a^*$;
    \smallskip
    \item[-] comultiplication $\Delta\colon \cH\to\cH\otimes_\cR\cH=\cF (G\times G\times M)$, 
    \beqn\Delta (\phi)(z, g_1,g_2)=\phi (z, g_1g_2)\,;
    \eeq
    \item[-] antipode $S\colon \cH\to\cH$, 
    $S (\phi)(z, g)=\phi (g^{-1}z, g^{-1})$,
\end{itemize}
for all $f\in\cR$, $\phi\in\cH$, $g,g_1,g_2\in G$, $z\in M$.
\end{example}


\noi\subsection{Lie-Rinehart pairs and their universal enveloping algebras.}\label{subsec:Lie-Rinehart}

\begin{deff}[Lie-Rinehart pair, \cite{Rinehart:1963}]\label{deff:Lie-Rinehart}
Let $\bk$ be a commutative ring, and let $\cR$ be a commutative $\bk$-algebra. Let $L$ be a Lie ring that is also an $\cR$-module. Suppose that we are given a Lie
ring and $\cR$-module homomorphism (an anchor) $\rho$ from $L$ to the $\cR$-derivations of $\cR$. 
If $x\in L$, we will denote the image of $x$ under this homomorphism by $r \mto x(r)$. 
Suppose finally that, for all $x,y\in L$ and $r\in\cR$,
\beqn
[x, ry] = r[x, y] + x(r)y
\eeq
We will call such $(\cR, L)$ an $\cR-$algebra or a Lie-Rinehart pair.
\end{deff}

\noi Let $(\cR,L)$ be a Lie-Rinehart pair. The left $\cR-$module $\cR\oplus L$ has a natural structure of a Lie ring, given by 
\beqn
\left[(r_1,x_1),(r_2,x_2)\right]=\left(x_1 (r_2)-x_2 (r_1),[x_1,x_2]\right)
\eeq
for any $r_1,r_2\in \cR$, $x_1, x_2\in L$. 
\begin{deff}[Universal enveloping algebra of a Lie-Rinehart pair, \cite{Rinehart:1963}]\label{deff:U_Lie-Rinehart}
Let $\UE^+ (\cR\oplus L)$ be the subalgebra of the universal enveloping algebra of $\cR\oplus L$, generated by the image of the natural inclusion $\cR\oplus L\hookrightarrow \UE^+ (\cR\oplus L)$. The universal enveloping algebra of the Lie-Rinehart pair is the quotient algebra
\beq\label{eq:u_Lie-Rinehart}
\UE (\cR, L)=\UE^+ (\cR\oplus L)/I\,,
\eeq
where $I$ is the two-sided ideal generated by elements  $r'(r+x)=r'r+r'x$ for all $r,r'\in\cR$, $x\in L$.
\end{deff}

\begin{example}[Lie algebroid, cf.~\cite{LSX,Moerdijk-Mrcun:2010}]\label{ex:Lie_algebroid}
A Lie algebroid $(E,[\cdot,\cdot],\rho)$ over a smooth base $M$ gives us a Lie-Rinehart pair with $L=\Gamma (E)$, $\cR=\cF (M)$, where the anchor homomorphism is induced on sections of $E$ by $\rho$ (we will denote it by the same symbol). The corresponding universal enveloping algebra is known as the universal enveloping algebroid of $E$. 
\end{example}

\noindent In the case where a Lie-Rinehart pair $(\cR,L)$ is fixed, we denote its universal enveloping algebra $\UE (\cR, L)$ simply as $\UE$. It has a natural $\cR-$bimodule structure with respect to the left- and right- multiplication on elements of the ring. By construction, 
$\UE$ is an associative (generally non-commutative) algebra. The tensor product of $m$ copies of $\UE$ over $\bk$
\beq\label{eq:k-tensor}
\sT^m_\bk \left(\UE\right)=
 \underbrace{\UE\otimes_\bk\cdots\otimes_\bk\UE}_m
\eeq
is also an associative unital algebra. Taking into account the fact that the Lie-Rinehart pair, and hence its universal algebra, are fixed, we will denote \eqref{eq:k-tensor} simply as $\sT^m_\bk$.
If $\cR$ is a unital ring, then so are $\UE$ and $\sT^m_\bk$, with the unit being  canonically induced by the one in $\cR$.  

\noi
Consider $\UE_m$, the tensor product of $m$ copies of $\UE$ over $\cR$ regarded as left $\cR-$modules; in other words, $\UE_m = \sT^m_\bk/I_{\sst R}$,
where $I_{\sst R}$ is a right $\sT^m_\bk-$ideal, generated by elements
\beq\label{eq:R-tensor_identities}
r u_1\otimes_\bk \cdots \otimes_\bk u_i \otimes_{\bk}\cdots \otimes_\bk u_m - u_1\otimes_\bk \cdots \otimes_\bk ru_i \otimes_{\bk}\cdots \otimes_\bk u_m
\eeq
for all $i=2,\ldots, m$, $r\in\cR$, and $u_1, \ldots, u_m\in \UE$. If $\cR$ is unital 
then in formula \eqref{eq:R-tensor_identities} it suffices to take all $u_i$ equal to $1$. Observe that $\UE_m$ has one canonical structure of a left $\cR-$module and $m$ different canonical structures of right $\cR-$modules, given by the right multiplication of elements $r\in\cR$ on different components of the tensor product: 
\beqn
u_1\otimes\cdots \otimes u_i\otimes\cdots u_m \mto
u_1\otimes\cdots \otimes u_i r\otimes\cdots u_m
\eeq
for $i=1,\ldots, m$. Note also that, while $\UE_m$ is a right $\sT^m_\bk-$module, it is no more a left $\sT^m_\bk-$module, since $I_{\sst R}$ is not a two-sided ideal. Thus $\UE_m$ is not an algebra. The situation can be "corrected" by replacing $\UE_m$ with $\overline{\UE}_m$, where the latter is defined as the locus of coincidence of $m$ right $\cR-$module structures on $\UE_m$. i.e. with the space of $\sum_\a u_{\a 1}\otimes \cdots \otimes \cdots u_{\a m}\in \overline{\UE}_m$, such that
    \beqn
    \sum_\a u_{\a 1}\otimes \cdots \otimes r u_{\a i}\otimes \cdots\otimes \cdots u_{\a m}=
    \sum_\a u_{\a 1}\otimes \cdots \otimes r u_{\a j}\otimes \cdots\otimes \cdots u_{\a m}
    \eeq
    for all $i,j=1, \ldots, m$. 
One can easily verify that $\overline{\UE}_m$ is an associative ring with the obvious multiplication rule
\beqn
\left( u_1\otimes\cdots \otimes\cdots u_m\right)
\left( v_1\otimes\cdots \otimes\cdots v_m\right)=
u_1v_1\otimes\cdots \otimes\cdots u_m v_m
\eeq
for all $u_1, v_1, \ldots, u_m, v_m\in \UE$. Indeed, let $\pi_{\sst R}$ be the projection $\sT^m_\bk\to \UE_m$. From 
\beq\label{eq:proj_prod}
\pi_{\sst R} (\bar{u}\bar{v})=\pi_{\sst R}(\bar{u})v
\eeq
for any $\bar{u}, \bar{v}\in\sT^m_\bk$ 
and the fact that left multiplication by elements of $\overline{\UE}_m$ preserve $I_{\sst R}\subset {\UE}_m$ we immediately deduce that $\pi_{\sst R} (\bar{u}\bar{v})$ depends only on $\pi_{\sst R} (\bar{u})$ and $\pi_{\sst R} (\bar{v})$.

\begin{rem}
    As far as we know, the above construction of $\overline{\UE}_m$ was proposed and used in \cite{kapranov07}, while the interpretation of $\overline{\UE}_m$ for the unital case that follows (see below) appeared in \cite{Moerdijk-Mrcun:2010}.
\end{rem}

\noi When $\cR$ is unital, there is an alternative construction of $\overline{\UE}_m$. Consider a subalgebra $\overline{\sT}^m_\bk$ 
of $\sT^m_\bk$ consisting of elements that preserve $I_{\sst L}$ under left multiplication. From \eqref{eq:proj_prod} it turns out that the action of elements from this subalgebra on the quotient space $\UE_m = \sT^m_\bk/I_{\sst R}$ is uniquely determined by their value at the unity $1^{\otimes m}$ and hence by their image under the projection map $\pi_{\sst R}$. Now $\overline{\UE}_m=\pi_{\sst R}\left(\overline{\sT}^m_\bk\right)$, such that $\pi_{\sst R}\colon \overline{\sT}^m_\bk\to \overline{\UE}_m$ is an epimorphism of rings. 

\begin{deff}[Left $\cR-$bialgebra]\cite{kapranov07}\label{def:left_bialgebra}
Let $\cR$ be a commutative $\bk-$ring. A left $\cR-$ 
bialgebra consists of
\begin{enumerate}
    \item A possibly noncommutative algebra $\cB$ containing $\cR$.
    \item A morphism of left $\cR-$modules $\e\colon \cB\to\cR$ which is a twisted ring homomorphism:
    \beqn
    \e (uv)=\e (u\e (v))
    \eeq
    \item A homomorphism of algebras $\Delta\colon \cB\to \cB\overline{\otimes}_\cR\cB$ which is identical on $\cR$, coassociative and and has the left and right counit properties with respect to $\e$. 
\end{enumerate}
Here $\cB\overline{\otimes}_\cR\cB$ is defined in the same way as $\overline{\UE}_m$ for $m=2$, i.e. as locus of coincidence of two right $\cR-$module structures on $\cB$.
\end{deff}

\begin{proposition}\cite{kapranov07}
 The universal enveloping algebra of a Lie-Rinehart pair admits a natural structure of a left $\cR-$bialgebra with $\e$ being the standard augmentation and $\Delta$ the coproduct induced by the standard one on the universal enveloping algebra of $L$, such that $\Delta (x)=x\otimes 1+1\otimes x$ for all elements of the Lie algebra $L$, combined with $\Delta (r)=r\otimes 1=1\otimes r$ for all $r\in \cR$. 
\end{proposition}

\begin{proposition}\label{prop:dual_to_ue}
Let $(\cR,L)$ be a Lie-Rinehart pair, $\UE=\UE (\cR, L)$ be its universal enveloping algebra. Consider $\cH=\Hom_\cR (\UE, \cR)$, the set of homomorphisms of left $\cR-$modules. Then $\cH$ has a canonical structure of a unital and counital commutative bialgebroid (see the remark after Definition \ref{def:Hopf_algebroid}), where the multiplication, comultiplicaion, left- and right- units are dual to the corresponding structures on $\UE$. In addition, X admits an antipode, which turns it into a Hopf algebroid.
\end{proposition}

\noindent The proof of the Proposition \ref{prop:dual_to_ue} would essentially reproduce the formal integration procedure for Lie algebroids written in \cite{kapranov07}. However, we did not find in the above paper an explicit construction of the antipode, which is defined in a more complex way than other structural morphisms, since it is not directly determined in terms of the universal enveloping algebroid (without involving auxiliary constructions). The latter is explained by the fact that the antipode in a groupoid interchanges left-invariant vector fields parallel to the source-fibers (corresponding to sections of the Lie algebroid) and right-invariant fields parallel to the target-fibers. In the case of action Lie algebroids, this obstacle is bypassed due to the flat Cartan connection, which singles out the elements of the Lie algebra as flat sections. In a more general case, an explicit construction of the antipode would require a more general (non-flat) connection. We are going to write more about this in another article.


\begin{example}[Action algebroid]\label{ex:action_alg}
Let $\g$ be a Lie algebra acting on a manifold $M$. Denote the action map by $\rho$, i.e. $\rho\colon \g\to \Gamma (TM)$ is a Lie algebra morphism. Let $E=\g\times M$ be the coresponding action algebroid, where the bracket on $\Gamma (E)$ and the anchor map are canonically extended from the bracket on $\g$ and the action map $\rho$ by use of the combination of the linearity and the Leibniz rule.

\noi Let $\cR=\cF (M)$ be the algebra of functions on the base $M$ and let $\cH=\Hom_\cR (\UE(E),\cR)=\cR\otimes_\bk\UE(\g)^* $ be the space of left $\cR-$modules morphisms $\UE (E)\to\cR$ (see Proposition \ref{prop:dual_to_ue}). The structure maps of the Hopf algebroid $\cH$ are explicitly defined\footnote{In what follows, we will denote the elements of the universal enveloping algebra $\UE(\g)$ and the universal enveloping algebroid $\UE(E)$ by the same letters, which is naturally motivated by the fact that one is a subspace in the other. But in order to avoid confusion, we nevertheless in each case will indicate the exact belonging of the elements.} as follows:
for all $u\in\UE (\g)$, $z\in M$, $f\in\cR$
\begin{itemize}[leftmargin = 2em]
    \item[--] left unit $\eta_L\colon \cR\to\cH$, $\eta_L (f)(z,u)=f(z)$, i.e. $\eta_L=\id\otimes_\bk 1$;
    \smallskip
    \item[--] right unit $\eta_R\colon \cR\to\cH$,
    $\eta_R (f)(z,u)=(uf)(z)$, where 
    \beqn
    (x_1\cdots x_m)f=\rho(x_1)\cdots\rho (x_m)(f)\,, \hspace{2mm}
    \forall x_1,\ldots, x_m\in \g
    \eeq
    \end{itemize}
   For all $\phi\in\cH$, $u,v\in\UE(\g)$, $z\in M$,
   \begin{itemize}[leftmargin = 2em]
    \item[--]  comultiplication $\Delta\colon \cH\otimes_\cR\cH=\cR\otimes_\bk\left(\UE(\g)\otimes\UE(\g)\right)^*$, 
    \beqn
    \Delta (\phi)(z, u,v)=\phi (z,uv)\,;
    \eeq
    \item[--] (algebroid) antipode 
    $S\colon \cH\to\cH$, 
    \beq 
    S (\phi)(u,z)=
    \sum_\a\eta_R \big(  \phi (S(u_\a^{(2)}), \cdot) \big)(u_\a^{(1)},z)\,,
    \eeq
    where (in Sweedler's notations)
    \beq\label{eq:Sweedler}\Delta (u)=\sum_\a u_\a^{(1)}\otimes u_\a^{(2)}\eeq
    and $\phi (S(u_\a^{(2)}), \cdot)$ is viewed as a function of $z\in M$. Here $S\colon\UE(\g)\to\UE(\g)$ in the standard antipode in the universal enveloping algebra \eqref{eq:antipode_UEg}.
\end{itemize}

\noi The proof that the introduced structures satisfy all the properties of a commutative Hopf algebroid, under the assumption that the action of the Lie algebra can be integrated to the action of the Lie group, follows from the Proposition \ref{prop:Hopf_epimorphism} below. The motivated reader, however, can verify all these properties by simple algebraic calculations without appealing to the indicated proposition.
\end{example}

\begin{rem}
 Below we will give a simple interpretation of the right unit map (in particular, this proves that $\eta_R$ is a monomorphism of algebras).
First of all, note that $\UE(E)$ is a $\cR-$bimodule. Then $\cH$, considered as the space of left $\cR-$module morphisms $\UE (E)\to\cR$, can be uniquely endowed with the structure of a right $\cR-$module by the formula
\beqn
\left(\phi f^{\sst R}\right) (u)\colon =\phi (uf)
\eeq
for any $f\in\cR$, $u\in \UE (E)$. On the other hand,
for all $x_1,\ldots, x_m\in \g$ one has
\beqn
x_1\cdots x_m f =\sum_{p=0}^m\sum_{\sigma\in Sh(p,m-p)}
\left(x_{\sigma(1)}\cdots x_{\sigma (p)}
\right)(f)
x_{\sigma(p+1)}\cdots x_{\sigma (m)}
\eeq
or, equivalently, in Sweedler's notations \eqref{eq:Sweedler} 
\beqn uf=\sum_\a u_\a^{(1)} (f) u_\a^{(2)}=\sum_\a \eta_R (f)\left(u_\a^{(1)} ,\cdot\right) u_\a^{(2)} \,.
\eeq
Thus for any $\phi\in\cH$
\beqn
\left(\phi f^{\sst R}\right) (u) =
\phi (f u) = \left(\eta_R (f)\otimes \phi\right)\left(\Delta (u)\right)=
\left(\eta_R (f)\phi\right)(u)
\eeq
which shows that $f^{\sst R}\phi$ coincides with the multiplication on $\eta_R (f)$ in $\cH$.
\end{rem}

\begin{proposition}\label{prop:Hopf_epimorphism}
Let $\mathscr{G}=M\times G$ be an action Lie groupoid, $E=\g\times M$ be the corresponding action Lie algebroid, and $\cH_{\sst\mathscr{G}}$ and $\cH_{\sst E}$ be the associated Hopf algebroids (see examples \ref{ex:action_grpd} and \ref{ex:action_alg}). Then there is a canonical epimorphism of Hopf algebroids $\cH_{\sst\mathscr{G}}\to \cH_{\sst E}$, defined as follows: $\cH_{\sst\mathscr{G}}\ni\tilde\phi \mto \phi\in \cH_{\sst E}$,
\beq  
\phi (z,u)=\left(\frac{\pt}{\pt\lambda_1}\cdots
\frac{\pt}{\pt\lambda_m}\right)_{{\la_1=\cdots=\la_m=0}}
\tilde\phi(z,e^{\la_1 x_1}\cdots e^{\la_m x_m})
\eeq
for $u=x_1\cdots x_m\in \UE(\g)$, $x_1,\ldots, x_m\in \g$.
\end{proposition}
\begin{proof}
The proof is straightforward computations.   $\square$
\end{proof} 

\begin{rem}\label{rem:Hopf_epimorphism}{\mbox{}}
The following points are here to interpret the meaning of ``straightforward computations'' from the proof of the proposition \ref{prop:Hopf_epimorphism} and in particular their link to the example \ref{ex:action_alg}:
\begin{itemize}[leftmargin = 2em]
     \item The structural mappings in the Example \ref{ex:action_alg} are obviously consistent with the structural mappings in the Example \ref{ex:action_grpd} and the epimorphism from the Proposition \ref{prop:Hopf_epimorphism}. On the other hand, Definition \ref{def:Hopf_algebroid} of a commutative Hopf algebroid generalizes the properties of the algebra of functions on a groupoid, and the action groupoid is certainly a groupoid, hence $\cH_{\sst\mathscr{G}}$ is ``beyond suspicion'' a Hopf algebroid. 
     Thus, even without knowing that all structures from the Example \ref{ex:action_alg} satisfy the properties of the Hopf algebroid, Proposition \ref{prop:Hopf_epimorphism} automatically convinces us of this thanks to the epimorphism and functoriality rules (whenever an action algebroid corresponds to an action groupoid).
     \smallskip
     \item The epimorphism described in Proposition \ref{prop:Hopf_epimorphism} is naturally interpreted as a morphism of Hopf algebroids dual to the embedding of a formal neighborhood of a unit section in a groupoid into the groupoid itself.
 \end{itemize}
\end{rem}

\noi The next proposition will combine the known construction of Harish-Chandra pairs for Lie algebras
with the examples \ref{ex:action_grpd} and \ref{ex:action_alg}. 
A more general approach to Harish-Chandra pairs for Lie algebroids and Lie groupoids, based on the idea that was first announced in \cite{Curitiba:2016}, will be proposed in our next article.

\begin{deff}\label{def:HC_action}
Let $E=\g\times M$ be an action algebroid, $E'=\h\times M$ be an action subalgebroid, where $\h\subset\g$ is a Lie subalgebra. Let $(H, \g)$ be a Harish-Chandra pair and
$\sG'=M\times H$ be an action groupoid integrating $E'$, such that 
\beq
h^*\rho(x)\left(h^{-1}\right)^* =\rho\left(
Adh^{-1}(x)\right)\colon \cF (M)\to\cF(M)
\eeq
for any $h\in H$ and $x\in\g$, where $h^*$ is the canonical pullback automorphism of $\cR=\cF(M)$ induced by $h$: $h^*(f)(z)=f(hz)$ for all $f\in\cR$, $z\in M$. We call $(\sG', E)$ a \emph{Harish-Chandra pair for action groupoid / algebroid}.
\end{deff}

\begin{proposition}
\label{prop:Harish-Chandra_for_action_gpd}
Define 
\beq
\cH=\Hom_{\UE (\h)}\left(\UE(\g), \cF (M\times H)\right)=
\{\phi\in \Hom_\bk (\UE (\g), \cF (M\times H)\,|\,
\\ \nonumber
\phi (z,h,xu)=
\left(\frac{\pt}{\pt \la}\right)_{\la=0}\phi (z,he^{\la x},u)\,, \forall x\in \h, u\in\UE (\g), z\in M\}\,.
\eeq
Then $\cH$
is a commutative Hopf algebroid with the following structure maps:
\begin{itemize}[leftmargin = 2em]
    \item[--] left unit:  $\eta_L (f)=f\otimes1\otimes 1$, i.e.
    $\eta_L (f) (z,h,u)=f(z)\,, \forall h\in H, u\in \UE(\g), z\in M$;
    \smallskip
    \item[--] right unit: $\eta_R =\left(\eta_R^{\sst \sG'}\otimes\id\right) \left(\eta_R^{\sst E}\right)$, where  $\eta_R^{\sst E}$ and $\eta_R^{\sst \sG'}$ are the right units for $E$ and $\sG'$, respectively;
    \smallskip
    \item[--] comultiplication: $\Delta (\phi)(z, h_1, u_1, h_2,u_2)=
    \phi (z, h_1h_2, Adh_2^{-1}(u_1)u_2)$, \\
    for all $h_1,h_2\in H$, $u_1,u_2\in\UE(\g)$, $z\in M$;
    \smallskip
    \item[--] antipode: $S(\phi)(z,h,u)=\left(S^{\sG'}\otimes\id
    \right)\left(S^{E}\right)_{(1,3)}(\phi)\left(z,h,Ad_h(u)\right)$, where the subscript $(1,3)$ means that we have inserted the identity into the second slot of a triple tensor product. Here we identified $\cH$ with $\cF(M)\otimes_\bk\cF (H)\otimes_\bk\UE (\g)$. 
\end{itemize}
\end{proposition}
\begin{proof}
    The proof is a straighforward combination of the proof of Hopf properties for action algebroid / groupoid, and Harish-Chandra pairs for groups / algebras. $\square$
\end{proof}

\noi The following relatively simple example illustrates the notion of the universal enveloping algebra of a Lie algebroid and its properties, including the Hopf algebroid properties for its dual, without resorting to action algebroids.

\begin{example}[Differential operators in the context of universal enveloping algebras and Hoph algebroids]\label{ex:diff_op}
Consider the tangent Lie algebroid $E=TM$ over $M$, whose sections are vector fields and the anchor map is the identity. One can verify that the universal enveloping algebra $\UE (TM)$ coincides with scalar differential operators on $M$, denoted by $\sD (M)$. The tensor product of $m$ copies of $\sD (M)$, being considered as left $\cF (M)-$modules, is isomorphic to the space of scalar $m-$linear polydifferential operators. The dual left $\cF (M)$-module $\cH =\Hom_{\cF (M)} (\sD (M), \cF (M))$ is isomorphic to $\sJ (M)$, the space of infinite jets of scalar functions on $M$. More precisely, given an element $\phi =\sum_l f^{(1)}_l j\left( f^{(2)}_l\right)$, where $f^{(1)}_l, f^{(2)}_l\in\cF (M)$ and $j (f)$ denotes the infinite jet prolongation of a function $f$, and a differential operator $\cp\in\sD (M)$, the explicit formula for the pairing is
\beqn
\cp\otimes \phi\mto \langle\cp,\phi\rangle=\sum_l f^{(1)}_l \cp\left( f^{(2)}_l\right)\,.
\eeq
The left- and right- $\cF(M)-$module structures on $\sJ (M)$ are determined by the multiplication on $f$ and $j(f)$, respectively, where $f\in\cF (M)$; it is compatible with the $\cF(M)-$bimodule structure on $\sD(M)$. Indeed, one has
\beqn
\langle f\cp,\phi\rangle = \langle\cp,f \phi\rangle =
f\langle\cp,\phi\rangle \\ \nonumber
\langle\cp f,\phi\rangle = \langle\cp,j(f)\phi \rangle
\eeq
for all $f\in \cF(M)$, $\phi\in \sJ (M)$, and $\cp\in\sD(M)$. The bialgebroid structure on $\cH$ is canonically induced by the one on differential operators, however, the antipode can be implicitly defined without direct use of the duality with differential operators as the only commutative algebra automorphism that maps elements of the form $f_1 j(f_2)$ to $f_2 j(f_1)$ for all $f_1, f_2\in \cF (M)$. Taking into account that $TM$ is the Lie algebroid for the pair groupoid $M\times M$, the Hopf algebroid $\cH=\sJ (M)$ constructed from $TM$ is identified with functions on the formal neighborhood of the diagonal\footnote{As far as we know, such an interpretation of the jet space should be attributed to Grothendieck \cite{Grothendieck:1965}.} in the Cartesian product $M\times M$, and the epimorphism from Proposition \ref{prop:Hopf_epimorphism} in this case is nothing more than taking the infinite jet of functions on $M\times M$ at the diagonal in the normal direction.
\end{example}


\section*{Instead of conclusion / perspectives}

In this paper we have revisited some algebraic constructions in the context of $\Z$-graded manifolds, this has permitted to fill the gaps in the general theory of integration of graded Lie algebras. We have also noticed that some of the techniques can be extended to the algebroid / groupoid setting. In the current study we have restricted the examples to ``not very infinite dimensional'' situations, but some of the constructions should be doable in the most general setting. 
We have seen that in some cases vector fields on an algebra can be transferred to the corresponding group preserving similar properties. A more complicated question of equipping a $\Z$-graded manifold with a differential structure (degree $1$ odd homological vector field) and describing the local structure of it is addressed in \cite{LKS}. For the current setting it will result in an interesting concept of so-called $Q$-groups.   

\subsection*{Acknowledgements.} We are thankful to MARGAUx --- the Nouvelle-Aquitaine Federation for Research in Mathematics --- the stay of A.K. in La Rochelle while writing this paper has been supported by its program ``Chaire Aliénor''. A.K.  also appreciates the support of the Faculty of Science of the University of Hradec Králové.

\appendix
\setcounter{equation}{0}

\section{Definitions, notations and conventions} \label{app:Z-cat}

\noi 
As usual, a \emph{$\Z$-graded vector space} ${V}$ decomposes into a direct sum:
  \begin{equation} \label{sumV}
  {V} =\bigoplus_{i\in \Z\setminus 0} V_i^{d_i}= \ldots\oplus V^{d_{-l}}_{-l} \oplus V^{d_{-l+1}}_{-l+1}\oplus \dots \oplus V^{d_{-1}}_{-1} \oplus \{0\} \oplus V^{d_{1}}_{1} \oplus \dots \oplus V^{d_{k}}_{k}\oplus\ldots,
  \end{equation}
where  $V$'s are vector spaces, the subscripts of $V_{\bullet}^{\bullet}$ denote the degree of elements of the respective subspace, and  the superscripts -- the dimension of it. Add to it an open set $V_0 \equiv \uU \subset \R^n$  to obtain ${U} = (V_0 \oplus V) = (\uU, V)$.

\noi The \emph{degrees} will be denoted by $\deg(\cdot) \in \Z$, and we will write the word 
  ``degree'' with no adjective, in contrast to for example ``polynomial degree'', that we will specify if needed. We will also use the notion of \emph{parity} of different objects: $p(\cdot) \in \Z_2 \equiv \{0, 1\}$. \\ It is responsible for commutation relations: 
$$
  a b = \varepsilon(a, b) b a,
$$   
where $\varepsilon(a, b) = (-1)^{p(a)p(b)}$ is the \emph{Koszul sign}. 
  
\begin{rem}
The parity has to be compatible with the degree, in the sense that the subsets of  even $(p = 0)$ and odd $(p = 1)$ objects are also consistently graded. We will comment on the geometric interpretation of this condition below, but for now it just means that parity depends on the degree. A frequent situation in the literature is $p(\cdot) = (deg(\cdot) \!\! \mod 2)$, but in general and in what follows it does not have to be like this.
\end{rem}

\textbf{Convention:} To stress the potential independence of degrees and parities, when the commutation relations are important we should in principal add the word or prefix ``super'' to the phrase. But for the sake of readability we will often omit it, when it does not lead to confusion. 
E.g. ``graded Lie super algebra'' will be just ``graded Lie algebra'', and ``graded vector superspace'' just ``graded vector space''.

\begin{deff}
A graded manifold $M$ is a topological space for which the sheaf of functions $\mathcal{O}(M)$ is locally modelled as functions on $(\uU, V)$.
\end{deff}

\begin{deff} 
We say that the graded vector space $V$ (resp. graded manifold $M$) is
\begin{itemize}[leftmargin = 2em]
\item  {of finite degree} if in (\ref{sumV}) $|i| < \infty$. 
 \item {of finite graded dimension} if $d_i = dim(V_i) < \infty, \forall i \in \mathbb Z$.
 \item {of finite dimension} if it is of finite graded dimension and of finite degree.
\end{itemize}

\end{deff}

\begin{deff} For a ring $R$, a $\Z$-graded $R-$module is the direct sum of $R-$modules:
$$
\sE = \bigoplus_{i\in{\Z \backslash \{0\} }} \sE_i\,.
$$
 $\sE$ is  \emph{of finite degree} if only a finite subset of modules $\sE_i$ are not equal to zero. \\
  If every $\sE_i$ is a free module of finite rank $d_i$, then $\sE$ is said to be \emph{of finite graded rank} $d_i$.\\
   If $d=\sum_{i\in{\Z\backslash \{0\}}}d_i<\infty$, then $\sE$ is called a graded $R-$module \emph{of finite (global) rank} $d$. \\ If $R$ is a field $\bk$, the word ``rank'' is to be replaced with ``dimension''.
\end{deff}

\noi Let $\sE$ be a $\Z$-graded $R-$module. We denote by $T^k \sE$ its $k-$th \emph{tensor power} over $R$
\beqn
T^k \sE = \underbrace{\sE \otimes \ldots \otimes \sE}_k \,.
 \eeq
 viewed as an $R-$module, and the direct sum of all tensors by $T(\sE)=\bigoplus_{k\ge 0} T^k \sE $. \\
 We define the \emph{symmetric powers} of $\sE$ over $R$: 
$$
 \Sym (\sE) := \mathrm{T}( \sE) / \left<  x\otimes y - \varepsilon(x, y) y \otimes x  \, | \, x, y\in \sE\right>\,.
$$
$\Sym(\sE)$ will be regarded as a free graded commutative $R-$algebra or $R-$coalgebra, depending on the situation.

\noi In \cite{AKVS} we have introduced the following increasing (decreasing) filtration of $\Sym(\sE^*)$ ($\Sym (\sE)$, respectively): \\
for all $p >0$, $F^p\Sym(\sE^*)$ is the graded ideal of the symmetric algebra, generated by elements of degree $\le -p$; $F_p \Sym (\sE)$ is the graded coideal of the symmetric coalgebra, cogenerated by
elements of degree $<p$, i.e. 
\beq\label{eq:cofiltration}
F_p \Sym (\sE) \colon = \{a\in \Sym (\sE) \, |\, \Delta (a)\in \Sym (\sE)\otimes\Sym (\sE)_{<p}\}\,,
\eeq
where $\Delta$ is the comultiplication in the algebra $\Sym (\sE)$.

\vskip 2mm\noindent Let $\cR$ be the graded projective limit of $\Sym(\sE^*)/F^p \Sym (\sE^*)$ for $p\to \infty$:
\beq\label{eq:proj_lim_R}
\cR \colon = \bigoplus_{i\in \Z} \varprojlim \Big(  \Sym(\sE^*)/F^p \Sym (\sE^*)\Big)_i\,,
\eeq
where
$$
\Big(  \Sym(\sE^*)/F^p \Sym (\sE^*)\Big)_i = 
\Sym(\sE^*)_i / F^p \Sym (\sE^*)_i\,.
$$

\noi
\begin{proposition}
For any $p\ge 0$, $ F_p \Sym (\sE)$
    has finite graded dimension.
\end{proposition}
\begin{proposition}
One has the canonical isomorphism of graded $R-$modules
\beq\label{isom:projlimit}
\cR\simeq \Big(\Sym (\sE)\Big)^*=
\uHom\left( \Sym (\sE),   R\right).
\eeq
\end{proposition}
And since $\Sym(\sE)$ is a graded $R-$coalgebra with the standard comultiplication
$$
\Delta\colon \Sym (\sE)\to \Sym (\sE\oplus\sE)=\Sym (\sE)\otimes \Sym (\sE)\,,
\hspace{2mm} \Delta (v)=v\otimes 1+1 \otimes v\, \hspace{2mm} \forall v\in\sE\,,
$$
we immediately deduce that $\cR$ is an $R-$algebra.

Considering the sheaf $\cO_M$ of functions on a graded manifold $M$ in view of the above filtrations, 
we can construct a canonical $\N^2-$graded manifold $\oM$, associated to $M$, the structure sheaf of which is $\cO_{\oM}\colon =\overline{\cO_M}$. $M$ and $\oM$ will thus be isomorphic in the category of $\Z$-graded manifolds. \\
In \cite{AKVS} we have proven the $\Z$-graded analog of the \emph{Batchelor's theorem}:
\begin{proposition} \label{Z-Batchelor}
There exists a non-canonical isomorphism of $\Z-$graded smooth manifolds between $M$ and the total space of $\cV=\cV_-\oplus \cV_+$, where $\cV_{\pm}$ are $\N$-graded vector bundles. 
That is $M=M_+\times_{M_0} M_-$, where the fibered product of graded manifolds is defined algebraically in terms of the corresponding sheaves of functions: $\cO_M=\cO_{M_+}\otimes_{\cO_{M_0}}\cO_{M_-}$. 
\end{proposition}

\bibliography{BibGraded}

\end{document}